\documentclass[11pt]{article}
\usepackage[utf8]{inputenc}
\usepackage{amsmath}
\usepackage{amssymb}
\usepackage{amsfonts}
\usepackage{amsthm}
\usepackage{mathrsfs} 
\usepackage{bm}
\usepackage{enumerate}
\usepackage{bbm}
\usepackage{tikz}
\usepackage{csvsimple}
\usepackage{centernot}
\usepackage{hyperref}
\usepackage[ruled,vlined]{algorithm2e}
\usepackage[margin=1in]{geometry}
\usepackage{comment}

\DeclareMathOperator{\lcm}{lcm}

\newcommand{\getColor}[1]{%
  \ifnum#1=0
    \colorlet{dotColor}{blue}%
  \else\ifnum#1=1
    \colorlet{dotColor}{red}%
  \else\ifnum#1=2
    \colorlet{dotColor}{green}%
  \fi\fi\fi
}

\hypersetup{
    colorlinks=true,
    linkcolor=blue,
    filecolor=magenta,
    urlcolor=blue,
    citecolor=blue,
    backref=true
}

\makeatletter
\newtheorem*{rep@theorem}{\rep@title}
\newcommand{\newreptheorem}[2]{%
\newenvironment{rep#1}[1]{%
 \def\rep@title{#2 \ref{##1}}%
 \begin{rep@theorem}}%
 {\end{rep@theorem}}}
\makeatother

\numberwithin{equation}{section}

\makeatletter
\renewenvironment{proof}[1][\proofname]{\par
  \vspace{-\topsep}
  \pushQED{\qed}%
  \normalfont
  \topsep0pt \partopsep0pt 
  \trivlist
  \item[\hskip\labelsep
        \itshape
    #1\@addpunct{.}]\ignorespaces
}{%
  \popQED\endtrivlist\@endpefalse
  \addvspace{6pt plus 6pt} 
}
\makeatother

\newtheorem{thm}{Theorem}
\newreptheorem{thm}{Theorem}

\newtheorem{result}{Result}[section]
\newtheorem{lem}[result]{Lemma}

\newtheorem{cor}[result]{Corollary}

\theoremstyle{definition}
\newtheorem{rmk}[result]{Remark}

\theoremstyle{remark}

\newcommand{\hide}[1]{}
\newcommand{\edit}[1]{}

\newcommand{\rough}[1]{}
\definecolor{darkgreen}{RGB}{75,150,75}
\newcommand{\review}[1]{}

\newcommand{\hides}[1]{}
\newcommand{\pub}[1]{}

\title{Three-color van der Waerden numbers grow super-exponentially}
\author{Jacob Fox\thanks{Department of Mathematics, Stanford University, Stanford, CA 94305. Email: {\tt jacobfox@stanford.edu}. Research supported by NSF Awards DMS-2452737 and DMS-215412.} 
\and Zach Hunter\thanks{Department of Mathematics, ETH, Z\"rich, Switzerland. Email: {\tt zach.hunter@math.ethz.ch}. Research supported in part by SNSF grant 200021-228014.}}
\date{\today}

\begin{document}

\maketitle

\begin{abstract}
For $k$ sufficiently large, we show that there is a three-coloring of the first $2^{k (\log^* k)/4}$ positive integers without any monochromatic $k$-term arithmetic progressions. Thus, the three-color van der Waerden number $w(k;3)$ grows faster than any exponential in $k$. We further prove a new lower bound on multicolor van der Waerden numbers which resolves a problem of Erd\H{o}s and Graham on canonical van der Waerden numbers.

\end{abstract}
\section{Introduction}\label{sec:intro}

The {\it van der Waerden number} $w(k;r)$ is the minimum positive integer $N$ such that every $r$-coloring of $[N]:=\{1,2,\ldots,N\}$ has a monochromatic $k$-term arithmetic progression. That these numbers exist is the celebrated theorem of van der Waerden \cite{vanderWaerden}. Schur had conjectured it about two decades earlier in order to prove another conjecture of his on consecutive quadratic residues (see \cite{BR,Soifer}).

The original proof gives an enormous upper bound on $w(k;r)$ that grows in the Ackermann hierarchy. Over the past century, estimating van der Waerden numbers has remained a fundamental open problem and mathematicians have speculated much about the true growth of $w(k;r)$. The first primitive recursive upper bound was proved by Shelah  \cite{Shelah} in 1988. This was subsequently improved by Gowers \cite{Gowers01} in 2001 through a new proof of Szemer\'edi's theorem that introduced higher-order Fourier analysis. It gave a constant tower height bound; specifically $w(k;r) \leq 2^{2^{r^{2^{2^{k+9}}}}}$. In the case where $k$ is fixed, a further improvement was recently proved by Leng, Sah, and Sawhney \cite{LSS} through quasi-polynomial bounds on the inverse theorem for Gowers uniformity norms \cite{LSS2}.

There has also been much interest in studying lower bounds on van der Waerden numbers. Erd\H{o}s and Rado \cite{ErRa52} in 1952 proved using a probabilistic argument that $w(k;r) \geq \sqrt{2kr^{k}}$. Berlekamp \cite{Berlekamp} in 1968 proved $w(p+1;2) \geq p2^p$ for prime $p$  using an algebraic construction. Erd\H{o}s and Lov\'asz \cite{erdos lovasz} in the paper that introduces the Lov\'asz local lemma proved that $w(k;r) \geq r^{k-1}/(4k)$. Since then, there have been several lower-order improvements on this bound using delicate enhancements, beginning with Szab\'o \cite{szabo} in 1990 through Kozik and Shabanov \cite{KS16} in 2016 (see \cite{KS16} for a full history). The first exponential improvement over the local lemma bound was given by Hunter \cite{Hunter} for fixed $r \geq 5$. In the case $r \geq 6$ is a fixed multiple of $3$, it gives a bound of the form $w(k;r) \geq (3^{r/3})^{(1-o(1))k}$. 
 
Since the 1930s, Erd\H{o}s was very interested in the growth of van der Waerden numbers. In fact, Erd\H{o}s and Tur\'an \cite{ErTu} in 1936 were motivated (see \cite{erdos1975}) to conjecture what is now known as Szemer\'edi's theorem in order to get better bounds for van der Waerden numbers. In particular, Erd\H{o}s \cite{Erdos1980,Erdos1981,Erdos1981b} offered \$500  to prove or disprove that $w(k;2)$ has super-exponential growth in $k$, that is, $\limsup_{k \to \infty} w(k;2)^{1/k} = \infty$. He guessed this holds \cite{erdos1975}, although he thought that this would be very difficult to prove. An opposing conjecture appearing in \cite{GRS,Monroe,XLL} is that $\lim_{k \to \infty} w(k;r)^{1/k}= r$ for all $r$.

We prove that van der Waerden numbers for three (or more) colors have super-exponential growth in $k$. That is, $\lim_{k \to \infty} w(k;r)^{1/k}= \infty$ if $r \geq 3$, so the opposing conjecture is false\footnote{Of course, the previously stated lower bound from \cite{Hunter} already disproved this conjecture when $r\ge 5$. But we establish this with a substantially stronger lower bound.} for $r \geq 3$.  We recall that the iterated logarithm $\log^* k$ is the number of times the logarithm function must be iteratively applied to $k$ before the result is at most $1$.

\begin{thm}\label{super-exponential}
For $k$ sufficiently large, $w(k;3) > 2^{k(\log^* k)/4}$.     
\end{thm}
 
An important ingredient in the proof of Theorem~\ref{super-exponential} is Lemma~\ref{cyclicdense}. It gives a novel probabilistic construction of a very dense subset of a large cyclic group that is far from containing a $k$-term arithmetic progression. One can then turn this into a $3$-coloring (using this set as one color class and the other elements are colored uniformly at random with the other two colors) that is far from containing a monochromatic $k$-term arithmetic progression and has a very dense color. This is an essential input for the random shifted product construction (Lemma~\ref{productlemma}) that allows one to get a larger construction (of a $3$-coloring of a cyclic group with desirable properties) from two smaller ones. In each application, the construction increases in size by a factor which is exponential in $k$ at the expense of a double exponential loss in a density parameter. Consequently, we can iterate this for roughly $(\log^* k)/2$ steps, leading to the bound in Theorem \ref{super-exponential}. While an early version of a random shifted product construction is key to the previous improvement in \cite{Hunter}, the new versions we utilize here are more efficient and allow for keeping the number of colors fixed through the iteration. 

It is also interesting to better understand van der Waerden numbers for more colors. The following theorem gives a lower bound on van der Waerden numbers when the number of colors is at least a large power of the logarithm of the forbidden monochromatic arithmetic progression length.

\begin{thm}\label{powervdw}
For each $\varepsilon\in (0,1)$ there exists $k_0(\varepsilon)$ so that the following holds. Suppose $k\ge k_0(\varepsilon)$ and $r\ge (\log k)^{3/\varepsilon}$, then $w(k;r) \geq r^{(1-\varepsilon) k(\log k)}$.
\end{thm}

\noindent Observe that $w(k;r)\ge N$ implies (by averaging) the existence of a color class $S\subset [N]$ with size at least $N/r$ and which lacks $k$-term arithmetic progressions. Thus, lower bounds on  van der Waerden numbers yield lower bounds for Szemer\'edi's theorem. In the other direction, if there is a set $S\subset [N]$ with $|S| \geq \varepsilon N$ with no $k$-term arithmetic progression, then we can use $r=(3/\varepsilon)\log_2 N$ random translates of $S$ to cover $[N]$ with positive probability, yielding $w(k;r) > N$. As noted in \cite{BLR}, it thus follows from the construction of Rankin \cite{Rankin} (which extends the case $k=3$ due to Behrend \cite{Behrend}, see also \cite{EHPS}), that there is $c_k>0$ such that $w(k;r) \geq 2^{c_k(\log r)^{\lceil \log k \rceil}}$. This is better than Theorem \ref{powervdw} if $r$ is {\it very large} with respect to $k$. Together with the Kelley-Meka theorem \cite{KM}, we have $w(3;r) = 2^{(\log r)^{\Theta(1)}}$.

However, when $r$ is not too large with respect to $k$ (i.e., when $r\le k^{O(1)}$), the authors are not aware of any way to construct a set $S\subset [r^{(1+\Omega(1))k\log k}]$ with density $1/r$ which does not contain a $k$-term arithmetic progression. This serves as a bottleneck for improving the exponent in Theorem~\ref{powervdw} by a factor of $1+\Omega(1)$. Moreover, any improvement to this bottleneck would immediately improve the lower bound above, by plugging the improved parameters into Theorem~\ref{masterpower}. 

\subsection{Canonical van der Waerden numbers} 

The {\it canonical van der Waerden number} $H(k)$ is the minimum $N$ such that every coloring of $[N]$ contains a $k$-term arithmetic progression which is monochromatic or rainbow (all different colors). The fact that these numbers exist was proved by Erd\H{o}s and Graham \cite{erdosgraham} through an application of Szemer\'edi's theorem. The  simple lower bound $H(k) \geq w(k;k-1)$ holds as any coloring with fewer than $k$ colors cannot contain a rainbow $k$-term arithmetic progression. Substituting in the lower bound which follows from the Lov\'asz local lemma, one gets $H(k) \geq \Omega(k^{k-2})$. Erd\H{o}s and Graham \cite{erdosgraham} posed as an open problem to determine if $H(k)^{1/k}/k \to \infty$ as $k \to \infty$ (see also \cite{bloom176}). 

The following theorem, which is a consequence of a better lower bound on $w(k;k-1)$, resolves the Erd\H{o}s-Graham problem. 

\begin{thm}\label{canonicallower}
We have $H(k) \geq k^{(1-o(1))k\log k}$.     
\end{thm}

\noindent We note that in very recent independent work, Bae \cite{Bae} proved $w(k;k-1)^{1/k}\ge (1-o(1))\frac{k^2}{e\log k}$ and consequently $H(k)\ge k^{(2-o(1))k}$, which also resolves the problem from \cite{erdosgraham} (despite getting a weaker bound compared to Theorem~\ref{canonicallower}). This was done by combining the well-known local lemma bound from \cite{erdos lovasz} and a recurrence proved in \cite{BCT}.

It is worth pointing out that we can already quickly resolve the Erd\H{o}s-Graham problem and verify that $H(k)=k^{\omega(k)}$ from Theorem \ref{super-exponential} (or from the construction in \cite{Hunter}) and a simple product-coloring (for details, see the discussion in Section \ref{section:conclude}). However, the lower bound one gets on $H(k)$ using this is significantly weaker than that given by Theorem \ref{canonicallower}.

\subsection{Further applications and future directions}

We note that parts of this paper are written in a higher level of generality than is strictly necessary to prove Theorems~\ref{super-exponential} and \ref{powervdw}. This is because we believe a number of these ideas to be of independent interest. One application we already mentioned is to canonical van der Waerden numbers. In a forthcoming companion note \cite{FH}, we shall present several further applications of the constructions presented here to obtain new bounds for different arithmetic Ramsey numbers. 

We suspect that the techniques developed here should be implementable to obtain new lower bound constructions for small van der Waerden numbers with at least three colors. See the Wikipedia page \cite{wikivdw} for current known lower bounds.

\noindent {\bf Organization.} In the next section, we present constructions of very dense sets without $k$-term arithmetic progressions in both the interval and cyclic setting. In Section \ref{sectionlll}, we present a simple consequence of the Lov\'asz local lemma that shows there are set colorings which are far from containing monochromatic $k$-term arithmetic progressions. In Section \ref{sectionrspc}, we provide several different variants of our random shifted product construction. In Section \ref{section:vdw}, we combine the various tools from the previous sections to obtain Theorems \ref{super-exponential} and \ref{powervdw} giving lower bounds on van der Waerden numbers. 

All logarithms are base $e$ unless otherwise specified. We also sometimes omit floor and ceiling signs when this has no effect on the argument. 

\noindent {\bf Disclosure.} The authors used Chat GPT Pro~5.5 for copyediting assistance during the final edits of this paper. All mathematical arguments and results in this paper were human-deduced.

\section{Very dense sets without long arithmetic progressions}\label{sec:dense}

Erd\H{o}s and Tur\'an \cite{ErTu} gave the following simple construction of a set of nonnegative integers with no long arithmetic progressions. For a prime $p$, let $T_p$ denote the set of nonnegative integers which, when written in base $p$, has no digit equal to $p-1$. The set $T_p$ has no $p$-term arithmetic progression. Indeed, if otherwise, letting $a,a+d,\ldots,a+(p-1)d$ be a $p$-term arithmetic progression with elements in $T_p$, consider the position $i$ of the least nonzero digit in base $p$ of the common difference $d$. As we go through the elements of the arithmetic progression, we get every element of $\{0,1,\ldots,p-1\}$ appearing exactly once in position $i$ in base $p$. One can also check that $T_p$ is precisely the greedy construction of a set of nonnegative integers which has no $p$-term arithmetic progression, so it is natural to suspect that it is rather dense (or at least at the beginning). Note that $T_p$ contains exactly $(p-1)^t$ elements from $\{0,1,\ldots,p^t-1\}$, which is the set of nonnegative integers with at most $t$ digits when written in base $p$. In particular, if $N=p^t$, we have $|T_p \cap \{0,1,\ldots,N-1\}|/N= (1-1/p)^t$. More generally, it follows from a simple induction argument on $t$ that if a positive integer $N$ is such that $N-1$ has at most $t$ digits when written in base $p$, then $|T_p \cap \{0,1,\ldots,N-1\}|/N \geq (1-1/p)^t$. 

The function $r_k(N)$ is the maximum size of a subset of $[N]$ which contains no $k$-term arithmetic progression. Szemer\'edi's theorem \cite{Szemeredi} states that for every fixed $k$, we have $r_k(N)=o(N)$. From the above discussion, we have the following corollary. 

\begin{cor}\label{szemcor}
If $p$ is prime, $t$ is a positive integer, and $N \leq p^t$, then $r_p(N) \geq (1-1/p)^t N$. 
\end{cor}

If $k=\Theta(\log N)$ we get $r_k(N)\ge (1-\Theta(1/\log N))N$ from Corollary \ref{szemcor} and Bertrand's postulate. In particular, in this case there are subsets of $[N]$ of density $1-o(1)$ that do not contain a $k$-term arithmetic progression. Such sets are a useful starting point for randomized shifted product constructions. However, for technical reasons, it is much more convenient to transition from the setting of discrete intervals to cyclic groups, so let us introduce some more notation. 

All groups we consider here are abelian. In an abelian group, a $k$-term arithmetic progression ($k$-AP for short) is a sequence $a,a+d,a+2d,\ldots,a+(k-1)d$ with $d$ nonzero. We say a subset of the group is \textit{$k$-AP-free} if it does not contain a $k$-AP. The following is an analogue of the Erd\H{o}s-Tur\'an construction in certain cyclic groups.

\begin{lem}
Consider primes $p_1,\ldots,p_s \leq k$ (not necessarily distinct). Let $N=\prod_{i=1}^s p_i$ and $m= \min_{i \in [s]} p_i$. There is a subset $S \subset \mathbb{Z}_N$ of size $|S|=\prod_{i=1}^s (p_i-1) \geq (1-1/m)^s N$ that is $k$-AP-free. 
\end{lem}
\begin{proof}
We can write $\mathbb{Z}_N$ as a product of cyclic groups, each of prime power order with distinct primes, and these primes are each at most $k$. Note that it suffices to prove the lemma in the case $N$ is such a prime power, as we can get the desired set $S$ in the general case by taking the product of the sets we get in the prime power case. So suppose $N=p^s$ with $p$ prime, $p \leq k$, and $s$ a positive integer. We view the elements of $\mathbb{Z}_N$ as $\{0,1\ldots,N-1\}$, and let $S \subset \mathbb{Z}_N$ consist of those elements that have no digit equal to $p-1$ when written in base $p$. The same argument as in the integer case gives that $S$ has no $p$-term arithmetic progression. As $p \leq k$, $S$ also has no $k$-term arithmetic progression. In this case, the size of $S$ is $(1-1/p)^s N \ge (1-1/m)^s N$. 
\end{proof}

Let $\delta_k(N):= \min_{M\le N}\{\frac{r_k(M)}{M}\}$. The reason for this choice is to make $\delta_k(N)$ monotone decreasing. Indeed, note that $r_k(N)/N$ is not monotone, as $\frac{r_k(2k-1)}{2k-1}= 1-1/(2k-1)>1-1/k=\frac{r_k(k)}{k}$. 

By Bertrand's postulate, for every integer $N \geq 2$, there is a prime $p$ satisfying $N/2 \leq p \leq N$. From this and  Corollary \ref{szemcor}, we have the following corollary. 

\begin{cor}\label{generaldeltakbound}
For all integers $k,t \geq 2$ and $N  \leq (k/2)^t$, we have $\delta_k(N) \geq (1-2/k)^t$. 
\end{cor}

For all $N$, we get two $k$-AP-free sets that together cover most of the elements $\mathbb{Z}_N$.

\begin{lem}\label{lemmatwosets0}
Let $k$ and $N$ be positive integers. There is a pair of disjoint $k$-AP-free subsets $S_1,S_2 \subset \mathbb{Z}_N$ such that $S_1 \cup S_2$ covers a density at least $\delta_k(N)$ of the group. 
\end{lem}
\begin{proof}
Let $N_1 = \lceil N/2 \rceil$. Let $S_1 \subset \{0,1,\ldots,N_1-1\}$ be a densest $k$-AP-free set, so $|S_1| \geq \delta_k(N)N_1$. Let $S_2 \subset \{N_1,\ldots,N-1\}$ be a densest $k$-AP-free set, so $|S_2| \geq \delta_k(N)(N-N_1)$. Observe that $S_1,S_2$ are disjoint and $|S_1 \cup S_2| \geq \delta_k(N)N$. After reversing the progression if necessary, since any arithmetic progression in $\mathbb{Z}_N$ has common difference with a representative that is a positive integer at most $N/2$, $S_i$ for $i=1,2$ is also $k$-AP-free when viewed as a subset of $\mathbb{Z}_N$.
\end{proof}
\begin{rmk} In a pedantic sense, Lemma~\ref{lemmatwosets0} can be close to sharp when $\delta_k(N)$ is close to $1$ (simply because for even $N$, any $S\subset \mathbb{Z}_N$ of size $>N/2$ contains a coset of $\{0,N/2\}$ by averaging, which is technically a $k$-AP). We highlight this to indicate that complications can arise when trying to construct very dense progression-free subsets of cyclic groups. In Subsection~\ref{subsect:cyclicfrom}, we will give an argument which allows us to handle arithmetic progressions that wrap around.\end{rmk}

\begin{cor}\label{coruniondense0} 
Let $k,s,N$ be positive integers. There are $2s$ subsets of $\mathbb{Z}_N$ that are each $k$-AP-free and cover all but a fraction at most $(1-\delta_k(N))^{s}$ of $\mathbb{Z}_N$. 
\end{cor}
\begin{proof}
Let $\rho:=(1-\delta_k(N))^{s}$. Let $S_1,S_2$ be the subsets of $\mathbb{Z}_N$ guaranteed by Lemma~\ref{lemmatwosets0}. Consider $s$ independent random translates of $S_1 \cup S_2$. The $2s$ sets we get (each a translate of $S_1$ or $S_2$) are each $k$-AP-free. Each element of $\mathbb{Z}_N$ has probability at most $\rho$ of not being in one of the sets. Hence, the expected density of elements not in these $2s$ sets is at most $\rho$. Therefore, there is an instance of these translates for which the density of elements not in these $2s$ sets is at most $\rho$.
\end{proof}

\begin{lem}\label{addscolors general}
    Suppose there is a coloring $c:\mathbb{Z}_N \rightarrow [r-2s]$ with no monochromatic $k$-APs. Then there is a coloring $c':\mathbb{Z}_N \rightarrow [r]$ with no monochromatic $k$-APs such that the density of elements whose color lies in $[r-2s]$ is at most $(1-\delta_k(N))^{s}$. 
\end{lem}
\begin{proof} By Corollary \ref{coruniondense0}, we can color all but $(1-\delta_k(N))^{s}N$ of the elements in $\mathbb{Z}_N$ with the $2s$ colors in $[r]\setminus [r-2s]$ so that there is no monochromatic $k$-AP. We color the remaining uncolored elements in $\mathbb{Z}_N$ as they are colored in $c$. The resulting coloring $c'$ of $\mathbb{Z}_N$ has the desired properties.
\end{proof}

In applying Lemma \ref{addscolors general}, it is helpful to use a lower bound on $\delta_k(N)$ like that in Corollary \ref{generaldeltakbound}.

\subsection{A sparse set robustly hitting arithmetic progressions} 

Throughout the rest of this section, we will find it more convenient to think in terms of complements. A {\it hitting set for $k$-term arithmetic progressions in $[N]$} is a subset of $[N]$ that contains at least one element in each $k$-term arithmetic progression in $[N]$. Note that the size of the smallest hitting set for $k$-term arithmetic progressions in $[N]$ is $N-r_k(N)$. So Szemer\'edi's theorem \cite{Szemeredi} says that, for $k$ fixed, as $N \to \infty$, every hitting set for $k$-term arithmetic progressions in $[N]$ has size $N-o(N)$. 

However, if $N$ is not too large as a function of $k$, one can hope to get much sparser hitting sets and even for the stronger notion in which the set is required to contain not only a single element but at least a $\delta$-fraction of each $k$-term arithmetic progression. The following key lemma shows that such sets exist even in certain cyclic groups as long as $N$ is at most exponential in $k$.

\begin{lem}\label{cyclicdense}
For each sufficiently small $\varepsilon>0$ and positive integer $m$ there is $\delta>0$ such that the following holds. For any integer $k \geq e^{72m/\varepsilon}$ and integer $N$ which is a product of $m$ distinct prime numbers in $[k,2.5^{k}]$, there is a subset $S \subseteq \mathbb{Z}_N$ that has density at most $\varepsilon$ such that $S$ contains at least a $\delta$-fraction of any $k$-term arithmetic progression in $\mathbb{Z}_N$. Furthermore, we may take $\delta=2^{-800m/\varepsilon}$. 
\end{lem}

The goal of this section is to prove Lemma \ref{cyclicdense}. The proof utilizes the following interval variant.

\begin{lem}\label{intervaldense}
For each sufficiently small $\varepsilon>0$ there is $\delta>0$ such that the following holds. For any integer $k \geq e^{36/\varepsilon}$ and integer $k \leq N \leq 2.5^k$, there is a subset $S \subseteq [N]$ with $|S| \leq \varepsilon N$ that contains at least a $\delta$-fraction of any $k$-term arithmetic progression in $[N]$. Furthermore, we can take $\delta =2^{-400/\varepsilon}$. 
\end{lem}
\begin{proof}
Let $\ell=\lceil k/101 \rceil$ and $Q$ denote the set of primes in the interval $(\ell,k]$. By the prime number theorem, $|Q| \geq 0.99k/\log k$. Let $t=\lfloor \varepsilon (\log k)/6\rfloor \geq \varepsilon (\log k)/7$, where the inequality uses $k \geq e^{36/\varepsilon}$. 

Independently for each prime $p \in Q$, choose $t$ distinct congruence classes modulo $p$ uniformly at random, and let $S_p$ be the elements in $[N]$ that are in at least one of these $t$ congruence classes. Let $S=\bigcup_{p \in Q} S_p$. 

{\bf Step 1: Density of $S$.} We first show that $|S| \leq \varepsilon N$. Indeed, $$|S|\le \sum_{p\in Q} |S_p|\le \sum_{p\in Q} (N/p+1)t = |Q|t+tN \sum_{p\in Q}1/p\le \varepsilon k /6+4.7\varepsilon N/6 < \varepsilon N.$$
To see the second to last inequality, let $$f(x)=\sum_{\substack{p~\textrm{prime} \\ p \leq x}} 1/p,$$ so $$\sum_{p \in Q} 1/p=f(k)-f(\ell)=\frac{(1+o(1))}{\log k}\sum_{n=k/101}^k 1/n = (1+o(1))\frac{\log(101)}{\log k}<4.7/\log k,$$ where the second equality holds by the prime number theorem, and the final inequality is due to $k$ being sufficiently large (which in turn holds as $\varepsilon$ is assumed to be sufficiently small).

{\bf Step 2: Discrepancy of $S$ in $k$-APs.} We next show that with positive probability the set $S$ has density at least $\delta$ in any $k$-term arithmetic progression in $[N]$. The proof is by a union bound. Fix a $k$-term arithmetic progression $P$ with elements in $[N]$ and let $d$ denote the common difference of $P$. Fix a subset $D$ of the arithmetic progression $P$ with size $u=\lfloor \delta k \rfloor$. 

If $d$ is a multiple of $s$ primes in $Q$, then the product of these primes is a factor of $d$ and hence $d \geq \ell^s$. As $N > d$, we have $s \leq \log_{\ell} N \leq k(\log 2.5)/\log \ell$. Let $Q'$ be the subset of $Q$ of primes $p$ that are not a factor of $d$. So $$|Q'| \geq 0.99k/\log k - k(\log 2.5)/\log \ell \geq k/(14\log k),$$using $\log\ell\ge\log k-\log(101)$ and the largeness of $k$. For each prime $p \in Q'$, we partition $P$ into subsets $P_i$ with $i \in \mathbb{Z}_p$ being those elements of $P$ that are congruent to $i \pmod p$. 

Each $P_i$ has size at least $\lfloor k/p \rfloor \geq k/(2p)$ as $p \leq k$. So the set $D$, of size $\delta k$, can contain all elements of at most $2\delta p$ of the subsets $P_i$. So the probability that $S_p \cap P \subseteq D$ is at most $(2\delta)^t$. As these choices are done independently for each prime $p \in Q'$, the probability that $S \cap P \subseteq D$ is at most $((2\delta)^t)^{|Q'|} \leq (2\delta)^{\varepsilon k/98}$. 

The number of choices of the progression $P$ is at most $N^2$. The number of choices for the subset $D$ of $P$ is ${k \choose u} \leq 2^k$. Hence, by the union bound, the probability that $S$ has density less than $\delta$ in some $k$-term arithmetic progression in $[N]$ is at most $$N^2 2^k (2\delta)^{\varepsilon k/98}<(2.5)^{2k}2^k(2^{1-400/\varepsilon})^{\varepsilon k/98}<1.$$

Since $S$ has the desired properties with positive probability, there exists such a desired set $S$, which completes the proof. 
\end{proof}

\subsection{From intervals to cyclic groups}\label{subsect:cyclicfrom}

We represent the elements of $\mathbb{Z}_N$ as $0,1,\ldots,N-1$. For simplicity of presentation, we will often identify an element in $\mathbb{Z}_N$ by its representative in $\{0,1,\ldots,N-1\}\subset \mathbb{Z}$. If $x,y \in \mathbb{Z}_N$, the {\it distance} between $x$ and $y$ is $\min(|x-y|,N-|x-y|)$, where $|x-y|$ is the absolute value of $x-y$ when we have replaced $x$ and $y$ by their identified integer representatives. We call a $k$-AP in $\mathbb{Z}_N$ {\it non-wrapping} if, when replacing its terms by its representatives in $\{0,1,\ldots,N-1\} \subset \mathbb{Z}$, it still forms a $k$-AP in $\mathbb{Z}$.  Otherwise, we say that $P$ {\it wraps around}. Equivalently, a $k$-AP $P$ in $\mathbb{Z}_N$ wraps around if it has no lift to a $k$-AP in $\mathbb{Z}$ contained in an interval $[(i-1)N,iN-1]$ for some integer $i$.

The next lemma shows that there is a small subset of $\mathbb{Z}_N$ that robustly intersects every $k$-term arithmetic progression in $\mathbb{Z}_N$ that wraps around. In the proof, we use the well-known fact that the least common multiple of the first $m$ positive integers is $e^{(1+o(1))m}$. This can easily be shown to be equivalent to the prime number theorem. 

\begin{lem}\label{intervaltocyclic} Let $\varepsilon>0$ be sufficiently small and $k>3^{1/\varepsilon}$. Suppose $N \geq k$ is prime. Then, there is a subset $T \subset \mathbb{Z}_N$ with density at most $4\varepsilon$ that has density at least $\varepsilon/4$ in any $k$-term arithmetic progression in $\mathbb{Z}_N$ that wraps around. 
\end{lem}
\begin{rmk}
    Though we present a simple physical proof, the below argument draws inspiration from known results in Fourier analysis (such as Weyl's equidistribution theorem). The fact we exploit is that either a $k$-AP should be ``well-distributed'' over $\mathbb{Z}_N$ (intersecting all dense intervals many times), or it has a common difference $d$ that is quite structured (with $d/N$ being very close to a ``bounded complexity rational'').
    \end{rmk}

\begin{rmk}
The assumption that $N$ is prime could be weakened to ``the smallest prime factor of $N$ is at least $\lfloor 1/\varepsilon\rfloor+1$'', but then we would need to count intersections with $k$-APs using multiplicity.
\end{rmk}

We will first need a simple lemma and some definitions. 

An {\it interval} in $\mathbb{Z}_N$ is a sequence $a,a+1,\ldots,a+k-1$ of consecutive elements. That is, an interval in $\mathbb{Z}_N$ is an arithmetic progression with common difference one. For a positive integer $i$, the {\it $i^{\textrm{th}}$ rotation} of an arithmetic progression $a,a+d,\ldots,a+(k-1)d$ in $\mathbb{Z}_N$ is the subsequence $(a+jd: (i-1)N \leq jd < iN~\textrm{and}~j \in \{0,1,\ldots,k-1\})$. It is a {\it full rotation} if $kd \geq iN$. Note that each rotation refers to a successive block of length $N$ traversed by the lifted progression $a,a+d,\ldots,a+(k-1)d$ in $\mathbb{Z}$. In an arithmetic progression $a,a+d,\ldots,a+(k-1)d$ in $\mathbb{Z}_N$, by  reversing the arithmetic progression if needed (which keeps the same terms but in reverse order), the (representative of the) common difference $d$ we may assume is a positive integer at most $N/2$. 

Say an interval $I\subset\mathbb{Z}_N$ is \textit{balanced} if $i\in I$ implies $-1-i\in I$. Note that if a non-empty interval $I \subset \mathbb{Z}_N$ is balanced, then $I=\mathbb{Z}_N$ or there is a positive integer $M<N/2$ such that the set of representatives of elements in $I$ is $\{0,1,\ldots,M-1\} \cup \{N-M,N-M+1,\ldots,N-1\}$, and $I$ has size $2M$. Thus a balanced interval in $\mathbb{Z}_N$ is centered at the cut between $N-1$ and $0$.

\begin{lem}\label{simplemma3} Let $M$ be a positive integer and $N$ be prime. Let $I \subsetneq \mathbb{Z}_N$ be a balanced interval of size $2M$.  
Let $P$ be a $k$-AP in $\mathbb{Z}_N$ with common difference $d \in [M]$. \begin{enumerate} 
\item If $kd \geq 9N$, then $|P \cap I| \geq |P||I|/(5N)$. 
\item If $P$ wraps around and $d \leq M/4$, then $|P \cap I| \geq |P||I|/(4N)$. 
\end{enumerate}
\end{lem}
\begin{proof}

We first prove part 1. All nonempty rotations of $P$ except possibly the last one are full. The number of full rotations of $P$ is $C:=\lfloor kd/N \rfloor \geq 9$. If the $i^{\textrm{th}}$ rotation is full, then it has length $\lceil iN/d \rceil - \lceil (i-1)N/d \rceil$, which is either $\lfloor N/d \rfloor$ or $\lceil N/d \rceil$. It follows that the number of elements of a rotation that is not full is at most the number of elements of any of the full rotations. The number of elements in any interval $I$ of $\mathbb{Z}_N$ in any full rotation of $P$ with common difference $d$ is at least $\lfloor |I|/d \rfloor-1$. Hence, the fraction of elements of a full rotation of $P$ that lie in $I$ is at least $$\frac{\lfloor |I|/d \rfloor-1}{\lceil N/d \rceil} \geq \frac{|I|/(3d)}{\lceil N/d \rceil} \geq (2/9)|I|/N,$$ where the first inequality uses $|I| \geq 2M \geq 2d$. Thus, the fraction of terms of $P$ that are in $I$ is at least $\frac{C}{C+1}(2/9)|I|/N \geq |I|/(5N)$. This completes the proof of part 1.

To prove part 2, if $P \subset I$, then we are done. So suppose otherwise. Since $P$ wraps around and $P \not \subset I$, some two consecutive terms of $P$ cross the cut between $N-1$ and $0$. Then, after possibly applying the symmetry $x \mapsto -1-x$ (which preserves $I$), there is a term $b \in \{0,\ldots,d-1\}$, and moving forward from $b$, the terms $b+jd$ with $b+jd \leq M-1$ are all in $P \cap I$. Together with the previous term $b-d \in P \cap I$, this gives at least $\lfloor M/d \rfloor$ elements in $P \cap I$. 

If $P$ has no full rotation, then $kd<N$, so $|P \cap I|/|P| \geq \lfloor M/d \rfloor/(N/d) \geq \frac{1}{2}M/N \geq |I|/(4N)$. 

If $P$ has a full rotation, then by the same argument as in part 1, the fraction of each full rotation of $P$ that lies in $I$ is at least $(\lfloor |I|/d \rfloor -1)/\lceil N/d \rceil \geq (\frac{7}{9}|I|/d)/(\frac{9}{8}N/d) > \frac{7}{10}|I|/N$. The last inequality used $|I| =2M \geq 8d$ and $N>|I| \geq 8d$. Hence, if $C \geq 1$ is the number of full rotations of $P$, then $|P \cap I|/|P| \geq \frac{C}{C+1}\frac{7}{10}\frac{|I|}{N} \geq \frac{|I|}{4N}$.
\end{proof}

\begin{proof}[Proof of Lemma \ref{intervaltocyclic}]
Let $\ell=\lcm(1,2,\ldots,\lceil 1/\varepsilon\rceil)$, so $\ell < 2.8^{1/\varepsilon}$. As $N$ is a prime and $N \geq k>1/\varepsilon$, $\ell$ is invertible in $\mathbb{Z}_N$. Let $M:=\lfloor \varepsilon N \rfloor$ and $I \subset \mathbb{Z}_N$ be the balanced interval of size $2M$. Let $\ell \cdot I$ be the dilation of $I$ by $\ell$.  We will prove that $T:=I \cup \ell \cdot I$ has the desired properties. Note that $|T| \leq |I|+|\ell \cdot I|= 4M \leq 4\varepsilon N$.

Consider a $k$-AP $P:=(a,a+d,\ldots,a+(k-1)d)$ in $\mathbb{Z}_N$. We may assume, after possibly reversing the progression, that the representative of its common difference is a positive integer $d \leq N/2$. Let $P' = \ell^{-1} \cdot P$ be the dilation of $P$ by $\ell^{-1}$, so that $P'$ is an arithmetic progression with common difference $d_1=\ell^{-1}d$. 

By Dirichlet's approximation theorem (i.e., the pigeonhole principle), there is a positive integer $q \leq 1/\varepsilon$ such that $qd_1$ has distance to $0$ in $\mathbb{Z}_N$ at most $\lfloor N/(\lfloor 1/\varepsilon \rfloor+1)\rfloor \leq \lfloor \varepsilon N \rfloor =M$. Let $d'$ be the distance of $qd_1$ to $0$ in $\mathbb{Z}_N$. We can thus partition the arithmetic progression $P'$ into arithmetic progressions $P_i$ with $i \in \{0,1,\ldots,q-1\}$, each with common difference $d'$ and each with at least $k':=\lfloor k/q \rfloor$ elements. Explicitly, set of elements of $P_i$ is $\{(a\ell^{-1}+id_1)+j(qd_1):0 \leq i+qj \leq k-1\}$ and we reverse the progression if necessary to get common difference $d'$. 

If $k'd' \geq 9N$, as $|I| =2M \geq 2d'$, then we can apply the first part of Lemma \ref{simplemma3} to each $P_i$ to see that the fraction of $P_i$ that is in $I$ is at least $|I|/(5N)$. Hence, in this case, the fraction of $P'$ in $I$ is at least $|I|/(5N)=2\lfloor \varepsilon N\rfloor/(5N) \geq \varepsilon/4$ using that $\varepsilon>0$ is sufficiently small.  Equivalently, the fraction of $P=\ell \cdot P'$ that lies in $\ell  \cdot I$ is at least $\varepsilon/4$. 

Otherwise, $k'd' < 9N$. Note that it follows from the definitions of $q$ and $\ell$ that $q$ divides $\ell$. Since $k'=\lfloor k/q \rfloor$ and $q \leq 1/\varepsilon < k/2$, we have $qk' \geq k/2$. Hence, $$(\ell/q)d' < (\ell/q)9N/k' \leq  (18\ell/k)N \leq \varepsilon N/5,$$ where we used $\varepsilon$ is sufficiently small, $\ell<2.8^{1/\varepsilon}$, and $k>3^{1/\varepsilon}$. 

As $qd_1\equiv \pm d' \pmod{N}$, multiplying by $\ell/q$ gives $d \equiv \pm (\ell/q)d' \pmod{N}$. Thus the distance from $d$ to $0$ is at most $(\ell/q)d'$, and since $1 \leq d \leq N/2$, we get $d \leq \varepsilon N/5 \leq M/4$ as $\varepsilon>0$ is sufficiently small. Applying the second part of Lemma \ref{simplemma3} to $P$, we obtain $|P \cap I|/|P| \geq |I|/(4N) \geq \varepsilon/4$. 

In any case, $T=I \cup \ell \cdot I$ contains at least a $\varepsilon/4$-fraction of $P$. \end{proof}

To prove Lemma~\ref{cyclicdense}, we effectively prove the $m=1$ case and use it to bootstrap to the general case. This bootstrapping uses the following lemma, which uses a simple product construction. 

\begin{lem}\label{crtproduct}
For $i \in [m]$, suppose $N_i$ is a positive integer with no prime factor less than $k$ for which there is $S_i \subseteq \mathbb{Z}_{N_i}$ which has density at most $\varepsilon$ and contains at least $\delta k$ elements from every $k$-AP in $\mathbb{Z}_{N_i}$. Then there is a subset $S \subset \mathbb{Z}_{N_1} \times \cdots \times \mathbb{Z}_{N_m}$ which has density at most $m\varepsilon$ and contains at least $\delta k$ elements in every $k$-AP in $\mathbb{Z}_{N_1} \times \cdots \times \mathbb{Z}_{N_m}$. 
\end{lem}
\begin{proof}
Let $S:=\{(a_1,\ldots,a_m) \in \mathbb{Z}_{N_1} \times \cdots \times \mathbb{Z}_{N_m}: a_i \in S_i~\textrm{for at least one}~i \in [m]\}$. The complement of $S$ has density $\prod_{i=1}^m \left(1-\frac{|S_i|}{N_i}\right) \geq (1-\varepsilon)^m \geq 1-m\varepsilon$, so $S$ has density at most $m\varepsilon$. 

Consider a $k$-AP $P$ in $\mathbb{Z}_{N_1} \times \cdots \times \mathbb{Z}_{N_m}$, and let $d=(d_1,\ldots,d_m)$ be its common difference, so $d_i \not = 0$ for some $i \in [m]$. The arithmetic progression projected to coordinate $i$ is a $k$-AP $P_i$ for which $S_i$ contains at least $\delta k$ of its elements. But then $S$ contains at least $\delta k$ elements of $P$. 
\end{proof}

\begin{proof}[Proof of Lemma~\ref{cyclicdense}]
Let $N=\prod_{i=1}^m N_i$, where $N_1,\ldots,N_m$ are distinct prime numbers in $[k,2.5^k]$. Note that, for each $i \in [m]$, as $N_i \geq k$ and $N_i$ is prime, $N_i$ has no prime factor less than $k$. 

For $i \in [m]$, apply Lemma~\ref{intervaldense} with $\varepsilon/(2m)$ in place of $\varepsilon$ and $N_i$ in place of $N$ to obtain that there is a subset $S_i$ of $[N_i]$ with density at most $\varepsilon/(2m)$ such that $S_i$ has density at least $\delta=2^{-400/(\varepsilon/2m)}=2^{-800m/\varepsilon}$ in any $k$-AP in $[N_i]$. By Lemma~\ref{intervaltocyclic} with $\varepsilon/(8m)$ in place of $\varepsilon$, there is a subset $T_i \subseteq \mathbb{Z}_{N_i}$ with density at most $\varepsilon/(2m)$ that has density at least $\varepsilon/(32m)$ in any $k$-AP in $\mathbb{Z}_{N_i}$ that wraps around. Note that $\delta \leq \varepsilon/(32m)$. The union $S_i \cup T_i$ (where $S_i$ is taken modulo $N_i$) has density at most $\varepsilon/m$ and has density at least $\delta$ in any $k$-term arithmetic progression in $\mathbb{Z}_{N_i}$.  

As $N_1,\ldots,N_m$ are distinct primes, they are pairwise relatively prime and the Chinese remainder theorem tells us that $\mathbb{Z}_N \cong  \mathbb{Z}_{N_1} \times \cdots \times \mathbb{Z}_{N_m}$. By Lemma \ref{crtproduct}, there is a subset $S \subset \mathbb{Z}_N$ that has density at most $m(\varepsilon/m)=\varepsilon$ and contains at least $\delta k$ elements in every $k$-AP in $\mathbb{Z}_N$.  
\end{proof}

\section{Lov\'asz local lemma consequences}
\label{sectionlll}

We first recall the symmetric version of the Lov\'asz local lemma \cite{Spencer77}. 

\begin{lem}[Symmetric version of the Lov\'asz local lemma] Let $B_1,B_2,\ldots,B_n$ be a sequence of ``bad'' events such that each event occurs with probability at most $p$ and such that each event is independent of all the other events except for at most $d$ of them. If $ep(d+1) \leq 1$, then there is a positive probability that none of the events occur. 
\end{lem}

The following lemma is a simple consequence of the Lov\'asz local lemma. 

\begin{lem}\label{lllconsequence}
Let $H$ be a $k$-uniform hypergraph with maximum degree $\Delta$, and  $r > s$ be positive integers. Let $0<\delta <1/2$. If $\Delta \leq \frac{1}{ekr}\left(\left(\frac{s\delta}{er}\right)^{\delta}r/s\right)^{k}$, then there is a way to assign a subset of $s$ colors from a set of $r$ colors to each vertex of $H$ such that for every edge $f$ and every color $i$, there are at most $(1-\delta)k$ vertices of $f$ whose color set contains $i$. 
\end{lem}
\begin{proof}
Consider a uniform random coloring of the vertices of $H$ by subsets of size $s$ of the set of $r$ colors. For any fixed color $i$, there is a probability $s/r$ of $i$ being in the color set of a given vertex, and each of these events are independent. For each edge $f$ of $H$, let $B_f$ be the ``bad'' event that $f$ has at least $(1-\delta)k$ vertices of one color. Let $t= \lfloor \delta k \rfloor$. Note that for $B_f$ to occur, there has to be $k-t$ vertices of $f$ whose color sets contain a common color. By the union bound, the probability of event $B_f$ is at most $$p:=r{k \choose t}(s/r)^{(1-\delta)k} < r(ek/t)^t(s/r)^{(1-\delta)k} \leq r\left(\left(\frac{er}{s\delta}\right)^{\delta}s/r\right)^{k}.$$ 

Event $B_f$ is independent of all but at most $d=k(\Delta-1)$ other bad events $B_{f'}$. By the symmetric version of the Lov\'asz local lemma \cite{Spencer77}, as $ep(d+1) \leq 1$, there is a positive probability that none of the bad events occur. Hence, there is a choice of the coloring which has the desired properties. 
\end{proof}

\begin{cor}\label{lllcor} Let $k,r,s \in \mathbb{N}$ with $r>s$. Let $0<\delta <1/2$. Let $G$ be an abelian group of order $N$ such that $N \leq \frac{1}{ek^2r}\left(\left(\delta/e\right)^{\delta}\left(r/s\right)^{1-\delta}\right)^{k}$ and $N$ has no prime factor less than $k$. Then there is a way to assign a subset of $s$ colors from a set of $r$ colors to each element of $G$ such that for every $k$-AP of $G$ and every color $i$, there are at most $(1-\delta)k$ elements of the $k$-AP whose color set contains $i$. 
\end{cor}
\begin{proof}
Consider the $k$-uniform hypergraph $H$ with vertex set $G$ whose edges are the $k$-APs of $G$. The edges are $k$-uniform since $N$ has no prime factor less than $k$. Each element can occupy at most $k$ positions for a $k$-AP, and there are less than $N$ common differences, so the maximum degree of $H$ satisfies $\Delta <kN$. By Lemma \ref{lllconsequence}, the desired coloring exists. 
\end{proof}

We have the following corollary in the case $r=2$ and $s=1$. In this case, we get such a coloring as long as $N \leq 2^{(1-\delta-o(1))k}$. 

\begin{cor}\label{lllcor2}
Let $G$ be an abelian group of order $N \leq 2^{(1-\delta-o(1))k}$ such that $N$ has no prime factor less than $k$. Here  $o(1) \rightarrow 0$ for $\delta>0$ fixed and $k \to \infty$. Then there is a two-coloring of $G$ such that any $k$-AP has at most $(1-\delta)k$ elements of the same color. 
\end{cor}

In particular, as long as $N \leq 2^{0.9k+1}$ has no prime factors less than $k$ and $k$ is sufficiently large, then there is a two-coloring of any abelian group $G$ of order $N$ such that any $k$-AP has at most $0.95k$ elements of the same color.

\section{Random Shifted Product Constructions}
\label{sectionrspc}
The next lemma is a basic product coloring lemma. From colorings of smaller groups without a monochromatic $k$-AP, we obtain a coloring of their product group without a monochromatic $k$-AP. 

\begin{lem}\label{simpleproductcoloring} Let $k$ be a positive integer and $H_1,H_2$ be finite abelian groups. Let $G=H_1 \times H_2$. Suppose for $i=1,2$ that $c_i:H_i \to [r_i]$ is an $r_i$-coloring of $H_i$ without monochromatic $k$-APs. Then the product coloring $c_1 \times c_2:H_1 \times H_2 \to [r_1] \times [r_2]$ also has no monochromatic $k$-APs. 
\end{lem}
\begin{proof}
The product coloring $c_1 \times c_2: H_1 \times H_2 \rightarrow [r_1] \times [r_2]$ assigns to the pair $(h_1,h_2) \in H_1 \times H_2$ the color $(c_1(h_1),c_2(h_2))$. As a $k$-term arithmetic progression has nonzero common difference, a $k$-term arithmetic progression in $H_1 \times H_2$ also satisfies for at least one $i \in \{1,2\}$ that when restricted to coordinate $i$ it is also a $k$-AP (so has nonzero common difference). As neither $c_1$ nor $c_2$ has a monochromatic $k$-AP, so does $c_1 \times c_2$. 
\end{proof}

While in Lemma~\ref{simpleproductcoloring} the size of the group we are coloring is the product of the sizes of the smaller groups, the number of colors also multiplies. This is a significant drawback, and we would prefer if the number of colors does not increase or at least not by much. This is where random shifted product constructions come to the rescue. 

As input, it is given for $i \in \{1,2\}$ an $r$-coloring $c_i$ of an abelian group $H_i$ that is far from containing a monochromatic $k$-AP and has a dense color class. It outputs an $r$-coloring $c$ of the product group $H_1 \times H_2$ that is also far from containing a monochromatic $k$-AP and has a dense color, but with weaker parameters. 

\begin{lem}\label{productlemma}
Let $r$ and $k$ be positive integers, $0<\varepsilon_1,\varepsilon_2,\delta_1,\delta_2 \leq 1$, and $H_1,H_2$ be finite abelian groups. Let $G=H_1 \times H_2$. Suppose for $i=1,2$ that $c_i:H_i \to [r]$ is an $r$-coloring of $H_i$ such that 
\begin{enumerate} 
\item For each $k$-AP $A$ in $H_i$ and every $j \in [r]$, there are at least $\delta_i k$ distinct elements of $A$ that are not of color $j$ in $c_i$, and   
\item The density of color $r$ in coloring $c_i$ is at least $(1-\varepsilon_i)$. 
\end{enumerate}
Let $\varepsilon = \varepsilon_1+\varepsilon_2-\varepsilon_1\varepsilon_2$ and $\delta=\min(\delta_1/3,\delta_2)$.  If $\varepsilon_2^{-2\delta_1 k/3}>r2^k|G|^2$, then there is a coloring $c:G \rightarrow [r]$ such that color $r$ has density at least $(1-\varepsilon)$, and each $k$-AP has at least $\delta k$ distinct elements not of any given color.
\end{lem}

\begin{proof}
We first describe how the random shifted product construction is defined. We color the fibers of the projection onto the second coordinate independently. So fix $x \in H_1$ and let $y_x$ be a uniformly random element of $H_2$. We color the fiber $x \times H_2$ according to the random shift by $y_x$ of $c_2$ with the colors permuted. Namely, for each $y \in H_2$ with $c_2(y)=r$, in coloring $c$ we assign $(x,y+y_x)$ the color $c_1(x)$. For each $i \in [r-1]$, in coloring $c$, all elements $(x,y+y_x)$ with $c_2(y)=i$ are assigned one of the $r-1$ colors in $[r] \setminus \{c_1(x)\}$ so that each element of $[r] \setminus \{c_1(x)\}$ is the color of exactly one of the sets $\{(x,y+y_x):c_2(y)=i\}$. 

We first show that color $r$ in $c$ always has density at least $(1-\varepsilon)=(1-\varepsilon_1)(1-\varepsilon_2)$. This follows as for each $x \in H_1$ for which $c_1(x)=r$, color $r$ in the fiber $x \times H_2$ has density at least $(1-\varepsilon_2)$, and the density of such $x \in H_1$ is at least $(1-\varepsilon_1)$. 

To complete the proof, we next prove that with positive probability, for each $k$-AP $A$ in $G$ and every $j \in [r]$, there are at least $\delta k$ distinct elements of $A$ that are not of color $j$ in $c$. This is proved by a union bound over all colors $j$ and $k$-APs $A$ in $G$. 

Consider a color $j$ and a $k$-AP $A$ in $G=H_1 \times H_2$ with common difference $d=(d_1,d_2) \in H_1 \times H_2$. If $d_1=0$, since the fibers $x \times H_2$ are colored according to a translate of $c_2$ with the colors permuted, then $A$ has at least $\delta_2 k \geq \delta k$ elements not of color $j$. So assume $d_1 \not = 0$. Let $A_1$ be the projection of $A$ onto the first coordinate. There are at least $\delta_1k$ elements of $A_1$ that in coloring $c_1$ are not of color $j$. Thus we may take $A' \subset A$ so that $a:=|A'|=\lceil \delta_1 k \rceil$ and the projection of $A'$ onto $H_1$ consists of distinct elements that receive colors different from $j$. Let $B \subset A'$ be a subset of $b:=\lfloor \delta k \rfloor$ elements of $A$. Each element of $A'$ has probability at most $\varepsilon_2$ of being assigned color $j$ by coloring $c$ independently of the other elements of $A'$ because they lie in different $H_2$-fibers, whose shifts were chosen independently. So the probability that every element of $A' \setminus B$ receives color $j$ is at most $\varepsilon_2^{a-b} \leq \varepsilon_2^{2\delta_1 k/3}$. 

The number of colors $j \in [r]$ is $r$. Since in an arithmetic progression in $G$, the first term and the common difference are elements of $G$, the number of $k$-APs in $G$ is at most $|G|^2$. Given a set of size $a$, the number of subsets of size $b$ is ${a \choose b} \leq 2^{a}\le 2^k$. By the union bound, the probability that there is a color $j$ and a $k$-AP $A$ in $G$ such that all but at most $\delta k$  distinct elements of $A$ are of color $j$ is at most $r|G|^2 2^k\varepsilon_2^{2\delta_1 k/3}<1$. Hence, with positive probability, the coloring $c$ has the desired properties. 
\end{proof}

The only way we know how to get much larger constructions through these methods involves adding additional colors at each iteration. The following version of a random shift product coloring lemma allows us to do this. It has as an input a set-coloring $c_1$ in which each element is assigned a subset of $s$ colors from a set of $r$ colors so that, in any $k$-AP, no color is in almost all of the $s$-sets of colors. It also has as input an $r$-coloring $c_2$ without monochromatic $k$-term arithmetic progressions and with the property that the union of the $s$ largest color classes is very large. The outputted coloring $c$ has $r$ colors. We will be able to use this, starting from an $(r-s)$-coloring without a monochromatic $k$-AP, to get larger constructions without a monochromatic $k$-AP at the expense of adding $s$ colors in the next application of the lemma.

Let $c_1:H_1 \rightarrow {[r] \choose s}$ be a coloring of $H_1$ by subsets of $[r]$ of size $s$. We view $[r]$ here as a set of $r$ colors so that each element of $H_1$ receives a subset of $s$ of the colors. For $i \in [r]$, {\it color class $i$}, denoted $c_1^{-1}(i)$, is the set of elements of $H_1$ whose color (which is a subset of $[r]$ of size $s$) contains $i$. The density of color $i$ is the fraction of $H_1$ that is in color class $i$. 

\begin{lem}[Shifted Product Construction, set-coloring-assisted] \label{productlemma4} Let $r,k,k',s$ be positive integers, $\varepsilon>0$, and $H_1,H_2$ be finite abelian groups. Let $G=H_1 \times H_2$. Let $c_1: H_1 \rightarrow {[r] \choose s}$ and $c_2: H_2 \rightarrow [r]$. Suppose the color classes of $c_1$ and $c_2$ are $k$-AP-free. 
Further suppose
\begin{enumerate} 
\item For each $k$-term arithmetic progression $P$ in $H_1$ and every color $i \in [r]$ of $c_1$, there are at least $k'$ distinct elements of $P$ that are not of color $i$ in $c_1$, and   
\item In coloring $c_2$, the set of elements with color in $[r] \setminus [r-s]$ has density in $H_2$ at least $(1-\varepsilon)$. 
\end{enumerate}
If $\varepsilon^{-k'} > r|G|^2$, then there is an $r$-coloring $c:G \to [r]$ without a monochromatic $k$-AP. 
\end{lem}
\begin{proof}

By permuting the colors of $c_2$, for each $S \in {[r] \choose s}$, there is a coloring $c_{2,S}$ which lacks monochromatic $k$-APs and the set of elements with color in $S$ has density at least $(1-\varepsilon)$ in $H_2$. 

We color the fibers of the projection onto the second coordinate independently. So fix $x \in H_1$ and let $y_x$ be a uniformly random element of $H_2$. We color the fiber $x \times H_2$ according to the shift by $y_x$ of $c_{2,S}$ with $S=c_1(x)$. That is, $c(x,y) = c_{2,c_1(x)}(y+y_x)$. 

We show that in this random coloring $c$ the expected number of monochromatic $k$-APs is less than one, and hence there is a choice of $c$ for which there is no monochromatic $k$-AP. Recall that the number of $k$-term arithmetic progressions in $G$ is at most $|G|^2$ and there are $r$ colors. So, it suffices to show that, for any given color $i$ and $k$-term arithmetic progression $Q$ of $G$, the probability that $Q$ is monochromatic of color $i$ is at most $\varepsilon^{k'}$. 

Consider a color $i$ and a $k$-term arithmetic progression $Q$ in $H_1 \times H_2$ with common difference $d=(d_1,d_2) \in H_1 \times H_2$. If $d_1=0$, since the fibers $x \times H_2$ are colored according to a translate of $c_2$ with the colors permuted, then $Q$ cannot be monochromatic. So assume $d_1 \not = 0$. Let $P$ be the projection of $Q$ onto the first coordinate. In coloring $c_1$, $P$ has at least $k'$ elements that are not in color class $i$, so there are at least $k'$ elements of $Q$ that each independently have probability at most $\varepsilon$ of being color $i$. This independence follows as $d_1 \not = 0$, so these $k'$ elements lie in distinct $H_2$-fibers, whose random shifts are chosen independently. Hence, $Q$ has probability at most $\varepsilon^{k'}$ of being monochromatic, completing the proof. 
\end{proof}

\begin{rmk}
    The conditions required for these random shifted product constructions are somewhat improvable. For simplicity, we used a union bound in the proofs of Lemmas \ref{productlemma} and \ref{productlemma4}. Instead, we could use the Lov\'asz local lemma. Recall that we gave an upper bound $p$ on the probability that a given $k$-AP $P$ whose common difference has nonzero first coordinate $d_1$ has at least $(1-\delta)k$ elements of the same color, and bounded the number of such $k$-APs by $|G|^2$. Instead, using the symmetric Lov\'asz local lemma, the event that any such $k$-AP $P$ has $(1-\delta)k$ elements of the same color is independent of all other such $k$-APs $P'$ unless there is an element of $P'$ whose first coordinate is the same as an element from $P$. The number of such $k$-APs $P'$ is less than $k^2|G|^2/|H_1|$. We could thus replace $|G|^2$ in the conditions in each of these lemmas by $ek^2|G|^2/|H_1|$. However, using the better versions appears to give only minor improvements in the applications. 
\end{rmk}

\section{Lower bounds on van der Waerden numbers}
\label{section:vdw}
In this section, we prove Theorems~\ref{super-exponential} and \ref{powervdw} giving lower bounds on van der Waerden numbers.  

The tower function $T_n(x)$ is $n$-fold exponential in $x$. That is, $T_0(x)=x$ and $T_{n+1}(x)=2^{T_n(x)}$.  

\begin{thm}\label{firstkey}
Let $\varepsilon>0$ be sufficiently small and $b$ be a positive integer. Suppose $k \geq T_{3b}(\varepsilon^{-1})$.
Let $q_1,q_2,\ldots,q_b$ be distinct primes in $[2^{0.9k},2^{0.9k+1}]$. For $j \in [b]$, let $n_j=\prod_{i=1}^{j} q_i$. Let $\delta_1=2^{-1600/\varepsilon}$. For $j>1$, let $\varepsilon_j=2^{-10j/\delta_{j-1}}$ and $\delta_j=2^{-800/\varepsilon_j}$. 

For each $j \in [b]$ there is a three-coloring of $\mathbb{Z}_{n_j}$ in which one color class has density at least $(1-\varepsilon)$ and each color class has density at most $(1-\delta_j)$ in any $k$-term arithmetic progression. \end{thm}

\begin{proof}
The fact that such primes $q_1,\ldots,q_b$ exist follows from the prime number theorem and the lower bound on $k$. 

The proof is by induction on $j$. In fact, we will prove that the dense color class has density at least $(1-(1-2^{-j})\varepsilon)$ in the three-coloring of $\mathbb{Z}_{n_j}$. 

To establish the base case $j=1$, let $S_1 \subset \mathbb{Z}_{n_1}$ which has density at least $(1-\varepsilon/2)$ and has density at most $(1-\delta_1)$ in any $k$-AP. Such a subset exists by taking the complement of the set given by Lemma \ref{cyclicdense} with $m=1$. By Corollary \ref{lllcor} applied with $r=2$, $s=1$, there is a two-coloring of $\mathbb{Z}_{n_1}$ in which each color class has density at most $(1-\delta_1)$ in each $k$-AP. By recoloring the elements of $S_1$  a third color, we get a three-coloring $C_1$ of $\mathbb{Z}_{n_1}$ in which one color has density at least $(1-\varepsilon/2)$ and each color class has density at most $(1-\delta_1)$ in each $k$-AP.

The induction hypothesis is that for some $j \geq 2$, we have found the desired three-coloring $C_{j-1}:\mathbb{Z}_{n_{j-1}}\rightarrow [3]$, which has density at least $(1-(1-2^{1-j})\varepsilon)$ in color $3$, and each color class has density at most $(1-\delta_{j-1})$ in each $k$-AP. 

By taking the complement of the set given by Lemma~\ref{cyclicdense}, there is a subset $S_j$ of $\mathbb{Z}_{q_j}$ with density at least $(1-\varepsilon_j)$ and whose density in any $k$-AP is at most $(1-\delta_j)$ (after checking the parameters\footnote{Specifically, we need that $q_j\ge \exp(60/\varepsilon_j)$. Noting $q_j\ge \exp(k/2)$, we just need to check $\varepsilon_j^{-1}\le T_{3b}(\varepsilon^{-1})/120$. 

Since $\varepsilon$ was taken sufficiently small, one has that $\delta_1^{-1}\le T_2(\varepsilon^{-1})$, hence a basic induction gives $\delta_j^{-1}\le T_3(\delta_{j-1}^{-1})\le T_{3j-1}(\varepsilon^{-1})$. Thus $\varepsilon_j^{-1}\le \delta_j^{-1}\le k/120$.}).

By Corollary \ref{lllcor} applied with $r=2$, $s=1$, there is a two-coloring of $\mathbb{Z}_{q_j}$ with colors $1,2$ in which each color class has density at most $(1-\delta_j)$ in each $k$-AP. By recoloring the elements of $S_j$ color $3$, we get a three-coloring $C_j':\mathbb{Z}_{q_j} \rightarrow [3]$ in which one color has density at least $(1-\varepsilon_j)$ and each color class has density at most $(1-\delta_j)$ in each $k$-AP.

We will next apply Lemma~\ref{productlemma}. We do this with $H_1:=\mathbb{Z}_{n_{j-1}}, c_1:= C_{j-1},H_2:=\mathbb{Z}_{q_j},c_2:= C'_j$ to get a coloring $c$ of $H_1\times H_2$ (which we shall call $C_j$). As $\varepsilon_j^{-2\delta_{j-1} k/3}>3 \cdot 2^k \cdot n_j^{2}$, we can indeed apply Lemma~\ref{productlemma} with these parameters.  As $n_{j-1}$ and $q_j$ are relatively prime, we have $H_1 \times H_2 = \mathbb{Z}_{n_{j-1}} \times \mathbb{Z}_{q_j} \cong \mathbb{Z}_{n_j}$. As $\varepsilon_j \leq \varepsilon/2^j$, color $3$ in the coloring $C_j$ has density at least $(1-\varepsilon_j)(1-(1-2^{1-j})\varepsilon)) \geq 1-(1-2^{-j})\varepsilon$. Also, in the coloring $C_j$, we are guaranteed that each color has density at most $(1-\delta_j)$ in each $k$-AP. This completes the proof by induction. 
\end{proof}

Theorem \ref{super-exponential} is a quick corollary of the previous theorem. 

\begin{proof}[Proof of Theorem \ref{super-exponential}]

Let $\varepsilon>0$ be a sufficiently small absolute constant. Let $b=\lfloor (\log^* k)/3.5\rfloor$. There are $b$ distinct primes $q_1,\ldots,q_b \in [2^{0.9k},2^{0.9k+1}]$ by the prime number theorem. Since $k$ is sufficiently large, $k \geq T_{3b}(1/\varepsilon)$. Let $n:=\prod_{j=1}^b q_j \geq 2^{0.9kb} \geq 2^{k(\log^* k)/4}$. 
By Theorem~\ref{firstkey}, there is a three-coloring of $\mathbb{Z}_n$ with no monochromatic $k$-AP. In particular, $w(k;3) > 2^{k(\log^* k)/4}$. 
\end{proof}

We now turn to establishing Theorem~\ref{powervdw}, getting much stronger bounds for the ``many color regime''. It will follow as a corollary from the next theorem.

\begin{thm}\label{masterpower}
    Suppose $k,r,m,s$ are positive integers with $k$ sufficiently large with $4ms \le r\le m^{k/2}$. Further suppose that $N<(2m)^{mk/3}$, and that $(1-\delta_k(N))^{-sk/100}>N$. There is $N' \geq N$ and an $r$-coloring of $\mathbb{Z}_{N'}$ without a monochromatic $k$-AP. Hence, $w(k;r)> N$.
\end{thm}

\begin{proof}
    For $i=0,1,\dots,m$, let $r_i := r+2s(i-m)$. We note that $r_i\ge r/2$ for each $i\ge 0$ and $r_m=r$. 

    Let $p_0,\dots,p_m$ be distinct primes with $(2m)^{k/3}\le p_i\le (2m)^{k/3+1}$ for all $i\in \{0,1,\dots,m\}$. Such primes exist by the prime number theorem and $k$ being sufficiently large. Now, for $i=0,\dots,m-1$, let $n_i:= \prod_{j=0}^i p_j$. Observe that $n_{m-1} \geq (2m)^{mk/3}>N$ by assumption. Hence, we may set $m'$ to be the smallest index where $n_{m'}>N$. Thus, it suffices to show there is an $r$-coloring $c_{m'}$ of $\mathbb{Z}_{n_{m'}}$ with no monochromatic $k$-AP. This will be constructed through a recursive process, where in step $i$ we construct $c_i:\mathbb{Z}_{n_i}\to [r_i]$ with no monochromatic $k$-AP. 

    By Corollary~\ref{lllcor} (since $r_i\ge r/2$ and $r_i/s\ge 2m$), for $i \in [m']$ there is a set coloring $c_{1,i}:\mathbb{Z}_{p_i} \rightarrow {[r_i] \choose s}$ such that for each color $j \in [r_i]$ and every $k$-AP $P$ in $\mathbb{Z}_{p_i}$, there are at least $k':=0.1k$ elements $x\in P$ with $j\not\in c_{1,i}(x)$. Indeed, the one thing to verify is that $$((0.1/e)^{0.1}(r_i/s)^{0.9})^k>(0.7 \cdot (2m)^{0.9})^k > (1.02)^k m^{k/2} (2m)^{k/3+1} >ek^2 r_i p_i.$$ 

We will repeatedly apply Lemma~\ref{productlemma4}. In step $i$, colorings $c_{1,i}:\mathbb{Z}_{p_i} \rightarrow {[r_i] \choose s}$ and $c_{2,i}:\mathbb{Z}_{n_{i-1}} \rightarrow [r_i]$ serve as $c_1$ and $c_2$ in the application of Lemma \ref{productlemma4}. As $\mathbb{Z}_{p_i} \times \mathbb{Z}_{n_{i-1}} \cong \mathbb{Z}_{n_i}$, this will output a coloring $c_i:\mathbb{Z}_{n_i} \rightarrow [r_i]$ without a monochromatic $k$-AP.

We next describe how we construct $c_0$, obtain $c_{2,i}$ from $c_{i-1}$ in step $i \in [m']$, and finally check that the conditions of Lemma \ref{productlemma4} are satisfied to construct coloring $c_i$ from colorings $c_{1,i}$ and $c_{2,i}$.

By Corollary \ref{lllcor} with $s=1$ and $\delta=1/k$, there is a coloring $c_0:\mathbb{Z}_{p_0} \rightarrow [r_0]$ without a monochromatic $k$-AP. If $m'=0$, we are done. Otherwise, we crudely have that $n_{m'}\le N^3$ (as $n_{m'-1}<N$ and $p_{m'}/p_0<p_0<N$).

Note $r_i=r_{i-1}+2s$. Also observe that $n_{i-1}\leq n_{m'-1} \leq N$. By Lemma~\ref{addscolors general} applied with $r=r_i$, $s=s$, and coloring $c_{i-1}:\mathbb{Z}_{n_{i-1}} \rightarrow [r_{i-1}]$, there is another coloring $c_{2,i}:\mathbb{Z}_{n_{i-1}} \rightarrow [r_i]$ which also has no monochromatic $k$-APs and for which the density of elements with color in $[r_{i-1}]$ is at most \[\varepsilon_i:=(1-\delta_k(n_{i-1}))^{s} \leq (1-\delta_k(N))^{s},\] hence by assumption we have $\varepsilon_i^{-0.1k} \geq N^{10} > n_i^3 \geq r_in_i^2$. Having satisfied all necessary assumptions, Lemma~\ref{productlemma4} implies that there is a coloring $c_i:\mathbb{Z}_{n_i} \rightarrow [r_i]$ without a monochromatic $k$-AP. Iterating up to $i=m'$ yields the result.
\end{proof}

We can now extract the following corollary, which implies Theorem~\ref{powervdw}.
\begin{cor}\label{lastcor}
    Fix $\alpha>0$ (not necessarily bounded by $1$) and $\varepsilon>0$. For all sufficiently large $k$, if $r\ge (\log k)^{2+\alpha}$, then $w(k;r)\ge r^{(1-c_\alpha-\varepsilon)k\log k}$, where $c_\alpha=2/(2+\alpha)$.
\end{cor}
\begin{proof}
    Without loss of generality, we may assume $\varepsilon$ is sufficiently small with respect to $\alpha$.

    Let $m:=\lceil 6 r^{2/k}\log r\log k\rceil$. Therefore, $2^{km/3}>r^{k\log k}$ and $m^{k/2}\ge r$. We next set $s:= \lfloor r/4m\rfloor \ge r^{1-2/(2+\alpha)}\frac{\log k}{25\log r}$, and $N := r^{(1-c_\alpha-\varepsilon)k\log k}$. Let $t:= \lfloor (1-c_\alpha-\varepsilon/2)k\log r \rfloor$. We have $N< k^t$ (since $k$ is large). By the prime number theorem, there exists a prime $p\in [(1-\varepsilon/10)k,k]$, hence\footnote{Using that for all $0 \leq x<1/2$, that $1-x    \ge \exp(-x-2x^2)$.} $$\delta_k(N)\ge \left(1-\frac{1}{p}\right)^t \geq \exp\left(-(p^{-1}+2p^{-2})t\right) \ge \exp\left(-(1+\varepsilon/5)t/k\right) \geq \exp(-(1-c_\alpha-\varepsilon/4)\log r)=:\delta.$$ 

Thus, $(1-\delta_k(N))^s\le \exp(-s \delta)= \exp(- \log k \cdot r^{\varepsilon/4} /25\log r)$. Note that $r^{\varepsilon /4}\ge 2500 (\log r)^2$ as $k$ and hence $r$ is sufficiently large with respect to $\varepsilon$. As $N\le \exp(k\log k\log r)$, together with the last two bounds, we can indeed apply Theorem~\ref{masterpower} to get the desired conclusion. 
\end{proof}

\section{Concluding Remarks}\label{section:conclude} 

The factor $r^{2/k}$ in the definition of $m$ in the proof of  Corollary~\ref{lastcor} is there to obtain the inequality $m^{k/2} \geq r$ needed to apply Theorem \ref{masterpower}. In order to obtain $w(k;r) \geq r^{(1-o(1))k\log k}$ for $r =(\log k)^{\Omega(1)}$, we could have alternatively removed the factor, proved the bound first for $r \leq 2^k$, and then repeatedly used a simple product bound (Lemma \ref{productvdwbound} below) to get a bound for larger $r$ using that $w(k;r)$ is increasing as a function of $r$. To obtain Lemma \ref{productvdwbound}, we first need a simple lemma that, by losing a factor of two in the number of colors, obtains a cyclic coloring without a monochromatic $k$-AP from an interval coloring without a monochromatic $k$-AP. 

\begin{lem}\label{colorcycfromint}
For each positive integer $N \leq 2w(k;r)-2$, there is a $(2r)$-coloring of $\mathbb{Z}_N$ that has no monochromatic $k$-AP. 
\end{lem}
\begin{proof}
  Let $N_1=\lceil N/2 \rceil$ and $N_2=\lfloor N/2 \rfloor$, so $N_1+N_2=N$ and $N_1,N_2<\min(\frac{N}{2}+1,w(k;r))$. So by definition of $w(k;r)$, there are $r$-colorings with distinct sets of colors of intervals of length $N_1$ and $N_2$ each without a monochromatic $k$-AP. Partition $\mathbb{Z}_N$ into two intervals of length $N_1$ and $N_2$; say the first interval is $I_1=\{0,\ldots,N_1-1\}$ and the second interval is $I_2=[N_1,N-1]$. Use the $r$-colorings on each of the intervals to get a $2r$-coloring of $\mathbb{Z}_N$. Any $k$-AP in $I_i$ for $i \in \{1,2\}$ is a $\mathbb{Z}$-arithmetic progression, so the $(2r)$-coloring has no monochromatic $k$-AP. 
\end{proof}

\begin{lem}\label{productvdwbound}
For integers $k \geq 3$ and $r_1,r_2 \geq 2$, we have $w(k;4r_1r_2) \geq w(k;r_1)w(k;r_2)$.
\end{lem}
\begin{proof}
For $i \in \{1,2\}$, let $w(k;r_i) \leq N_i \leq 2w(k;r_i)-2$ be such that $N_1,N_2$ are distinct primes (such primes exist by a strengthening of Bertrand's postulate that follows from Erd\H{o}s' proof). By Lemma \ref{colorcycfromint}, there is a $2r_i$-coloring of $\mathbb{Z}_{N_i}$ without a monochromatic $k$-AP. By Lemma \ref{simpleproductcoloring}, there is a $(4r_1r_2)$-coloring of $\mathbb{Z}_{N_1N_2}\cong \mathbb{Z}_{N_1} \times \mathbb{Z}_{N_2}$ without a monochromatic $k$-AP. 
\end{proof}

In the introduction, we stated that the canonical van der Waerden number satisfies $H(k) = k^{\omega(k)}$ easily follows from our proof that $w(k;3)$ has super-exponential growth and from a simple product coloring. 
Indeed, repeatedly applying Lemma \ref{productvdwbound} with the bound $w(k;3) \geq 2^{k(\log^* k)/4}$ from Theorem \ref{super-exponential} gives that there is an absolute constant $c>0$ such that $w(k;r) \geq r^{ck \log^* k}$ for $r \geq 3$ and all $k$. Hence, $H(k) \geq w(k;k-1) \geq k^{ck \log^* k}$ (with a slightly smaller $c$). A better bound in terms of the constant $c$ comes from directly using Lemma \ref{simpleproductcoloring} on the three-colorings of the cyclic groups that we get with no monochromatic $k$-AP. Of course, this is still considerably weaker than the bound $w(k;r) \geq r^{(1-o(1))k\log k}$ from Theorem \ref{powervdw} we get when $r$ is super-polylogarithmic in $k$.

\end{document}